\documentclass[letterpaper, 10 pt, conference]{ieeeconf}  

\IEEEoverridecommandlockouts                              
\title{\LARGE \bf
Construction of the Barrier for Reach-Avoid Differential Games in Three-Dimensional Space with Four Equal-Speed Players
}

\author{Rui Yan, Zongying Shi and Yisheng Zhong\\Tsinghua National Laboratory for Information Science and Technology (TNList)
\thanks{This work was supported by the National Natural Science Foundation of China under Grants 61374034 and 1210012.}
\thanks{R. Yan, Z. Shi, and Y. Zhong are with the Department of Automation, Tsinghua University, Beijing 100084, China email: yr15@mails.tsinghua.edu.cn and \{szy, zys-dau\}@mail.tsinghua.edu.cn}%
}

\usepackage{amsmath}
\usepackage{amssymb}

\usepackage{bm}
\usepackage[compress]{cite}

\usepackage{graphicx}
\usepackage{indentfirst}

\usepackage{subfigure}

\usepackage{lineno}
\usepackage{multirow}
\usepackage{amsfonts}
\usepackage{textcomp}
\usepackage{stfloats}
\usepackage{threeparttable}
\usepackage{hyperref}
\usepackage{float}
\usepackage{indentfirst}
\usepackage{mathrsfs}
\usepackage[ruled,vlined]{algorithm2e}

\newtheorem{thom}{\textbf{Theorem}}

\newtheorem{lema}{\textbf{Lemma}}

\newtheorem{pbm}{\textbf{Problem}}

\begin{document}

\maketitle
\thispagestyle{empty}
\pagestyle{empty}

\begin{abstract}
This paper considers a reach-avoid differential game in three-dimensional space with four equal-speed players. A plane divides the game space into a play subspace and a goal subspace. The evader aims at entering the goal subspace while three pursuers cooperate to prevent that by capturing the evader. A complete, closed-form barrier for this differential game is provided, by which the game winner can be perfectly predicted before the game starts. All possible cooperations among three pursuers are considered and thus the guaranteed winning for each team is a prior. Furthermore, an algorithm is designed to compute the barrier for multiple pursuers of any numbers and any initial configurations. More realistically, since the whole achieved developments are analytical, they require a little memory without computational burden and allow for real-time updates, beyond the capacity of traditional Hamilton-Jacobi-Isaacs method.

\end{abstract}

\section{INTRODUCTION}

Pursuit-evasion differential games have been studied extensively, and many extensions involving game setups, game scales and specific occasions, have been proposed over the years \cite{Liu6907859,BOPARDIKAR20091771,Oyler2016Pursuit,MACIAS2018271}. Different from the classical pursuit-evasion game formulation, an interesting pursuit-evasion game, called reach-avoid differential game, is proposed in which each team strives to drive the system state into his own target set, while avoiding the target set of the opponent. This definition implies the qualitative solution, namely, which team can achieve its goal. The qualitative solution is a necessary step to solve the more complicated quantitative game whose objective is to optimize specific continuous payoff functions. This game has plenty of applications in areas such as collision avoidance, path planning, security, reachability analysis and confrontational situations \cite{Mylvaganam7909033,Ruiz2013Time,Yan2017Escape}.

The pioneer work on reach-avoid games may be found in \cite{GETZ1979421}, where the so-called two-target qualitative differential game was presented. It mainly focuses on the computation of barrier, a core concept in the game of kind \cite{Is1967Diff}, by employing the classical Isaacs' method \cite{Is1967Diff}. The key part of this method is to determine the boundary of the usable part of the target sets and the initial states for backward integration. 

In the references \cite{Mitchell2005A,Margellos2011Ham,Fisac2015Reach}, the term, reach-avoid differential games, was first introduced and originated from the area of reachability analysis. By defining a value function merging the payoff function and discriminator function with minimax operation, the barrier, or called boundary of reach-avoid set, can be located by finding the zero sublevel set of this value function \cite{Chen104941,8619382,Chen8267187}. This approach involves solving Hamilton-Jacob-Isaacs partial differential equation, and thus bears the curse of dimensionality.

For certain games and game setups, geometric method is widely employed and provides both qualitative and quantitative analysis about the game winner and optimal strategies. For example, Voronoi diagram and Apollonius circle are used for generating strategies in multiplayer pursuit-evasion games, such as Voronoi area minimization \cite{Zhou2016Cooperative,Pierson2017Int} and the closest point on Apollonius circle \cite{YAN8279644}. In \cite{Zou8588382}, the line-of-sight was proposed to address the target tracking problems in an environment with obstacles. Reference \cite{ChenMo2016Multiplayer} designed a constructive way to approximate the reach-avoid set by successively creating a number of straight lines called paths of defense.  

Recently several attractive reach-avoid games on specific regions were analyzed. The work in \cite{Garcia8619026} revisited the Capture-the-Flag differential game in convex domain by deriving the state-feedback strategy for each player and computing the value function. In \cite{Yan2017Defense}, the authors presented a defense game in a circular region and constructed the barrier analytically. Also in a circular region, by specifying the form of feedback control law, a confinement-escape problem was investigated in \cite{Li7503136}. 

This paper considers a three-pursuer-one-evader reach-avoid differential game in three-dimensional space. All players have the same speed and the game space $\mathbb{R}^3$ is separated by a plane into two subspaces: play subspace and goal subspace. The evader initially lying in the play subspace, attempts to enter the goal subspace by penetrating the splitting plane, while three pursuers aim at preventing that by capturing it. Actually, from another side, this game can also be viewed that an evader tries to escape from a subspace through its boundary which is a plane, while avoiding moving obstacles formulated as pursuers, especially, three dynamic obstacles are considered here. 

The main contributions of this paper are as follows: First, the condition to determine which players will contribute to the construction of barrier, is given. Second, we consider all possible cooperations among pursuers and construct the corresponding barrier by which the space of initial configuration is divided into two winning subspaces: pursuer winning subspace and evader winning subspace. Third, extensions to multiple pursuers are performed. Fourth, analytical solutions are obtained which allow for real-time computations and updates.  

The rest of this paper is organized as follows. Section \ref{problemdescriptionsection} states the problem. In Section \ref{prelinimaries}, several preliminary results are presented. The barrier and winning subspaces are analytically computed in Section \ref{winningsubandbarriersec}. Finally, Section \ref{conclusion} concludes the paper.

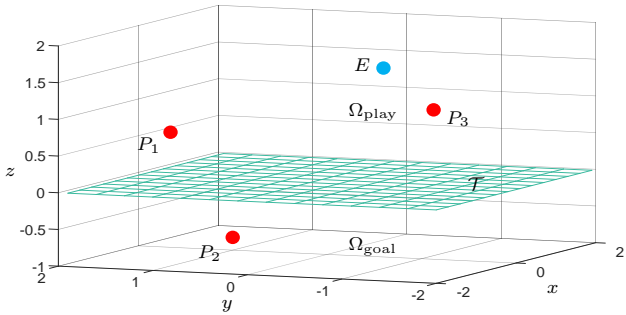
\begin{figure}\centering
\graphicspath{{figure_final/}}
\includegraphics[width=83mm,height=42mm]{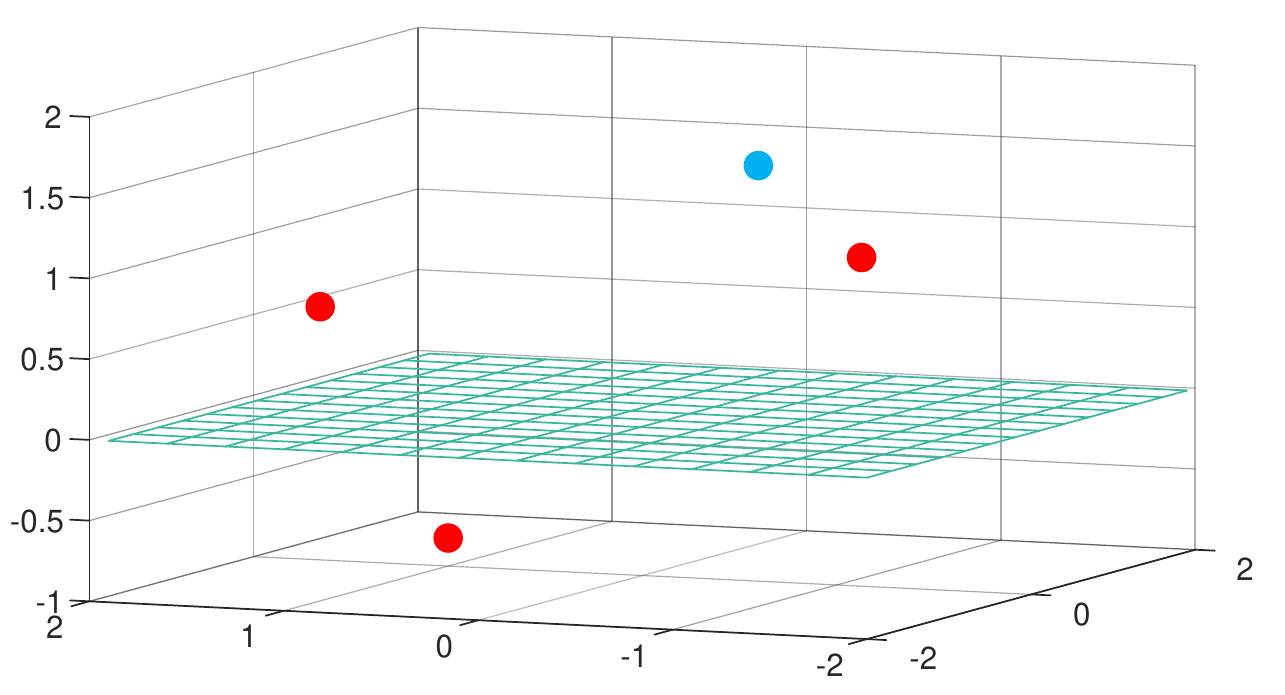}
\put(-110,22){$\scriptstyle{\Omega_{\rm goal}}$}
\put(-110,73){$\scriptstyle{\Omega_{\rm play}}$}
\put(-65,45){$\scriptstyle{\mathcal{T}}$}
\put(-190,60){$\scriptstyle{P_1}$}
\put(-167,19){$\scriptstyle{P_2}$}
\put(-73,70){$\scriptstyle{P_3}$}
\put(-108,90){$\scriptstyle{E}$}
\put(-35,6){$\scriptstyle{x}$}
\put(-158,0){$\scriptstyle{y}$}
\put(-240,50){$\scriptstyle{z}$}
\caption{Three-pursuer-one-evader reach-avoid differential games in three-dimensional (3D) space, where three pursuers (red) cooperate to capture an evader (blue) before it goes through the green plane $\mathcal{T}$ and gets into the goal region $\Omega_{\rm goal}$.\label{fig:1}} 
\end{figure}

\section{Problem Description}\label{problemdescriptionsection}
\subsection{Reach-Avoid Games with Four Equal-Speed Players}
A three-pursuer-one-evader reach-avoid differential game in three-dimensional space is considered. The game is played in $\mathbb{R}^3$, in which a plane $\mathcal{T}$ splits the game space $\mathbb{R}^3$ into two disjoint subspaces $\Omega_{\rm goal}$ and $\Omega_{\rm play}$, and their mathematical descriptions are given as follows:
\begin{equation}\begin{aligned}\label{ori_tar}
&\mathcal{T}=\big\{\mathbf{z}\in\mathbb{R}^3|K^\mathsf{T}\mathbf{z}=b\big\},\Omega_{\rm{goal}}=\big\{\mathbf{z}\in\mathbb{R}^3|K^\mathsf{T}\mathbf{z}<b\big\}\\
&\Omega_{\rm{play}}=\big\{\mathbf{z}\in\mathbb{R}^3|K^\mathsf{T}\mathbf{z}>b\big\}
\end{aligned}\end{equation}
where $K\in\mathbb{R}^3$ and $b\in\mathbb{R}$ are the known parameters, and $K$ is a nonzero vector. Three pursuers $P_1,P_2,P_3$ and one evader $E$, assumed to be four mass points in $\mathbb{R}^3$, can move freely with simple motion, namely, they are able to change the directions of their motion at each instant of time. Four players are assumed to have the equal unit speed. The evader $E$ is considered to have been captured as soon as his distance from the closest pursuer becomes equal to zero. The evader, starting from $\Omega_{\rm{play}}$, aims at reaching $\Omega_{\rm{goal}}$ without being captured, while three pursuers, initially distributed in any positions of the game space, cooperate to guard $\Omega_{\rm{goal}}$ by capturing $E$. Thus, these two subspaces $\Omega_{\rm{goal}}$ and $\Omega_{\rm{play}}$ shall be called goal subspace and play subspace, respectively. We call $\mathcal{T}$ as target plane (TP). The evader wins if its state can reach $\Omega_{\rm{goal}}$ before captured, while three pursuers win if $E$ can be captured in $\Omega_{\rm{play}}$. The game components are shown in Fig. \ref{fig:1}.

Define the unit control set $\mathcal{U}=\{\mathbf{u}\in\mathbb{R}^3|\|\mathbf{u}\|_2=1\}$, where $\|\cdot\|_2$ stands for the Eculidean norm in $\mathbb{R}^3$. Denote the positions of $P_i$ and $E$ at time $t$ in $\mathbb{R}^3$ by $\mathbf{x}_{P_i}(t)=\big(x_{P_i}(t),y_{P_i}(t),z_{P_i}(t)\big)$ and $\mathbf{x}_E(t)=\big(x_{E}(t),y_{E}(t),z_{E}(t)\big)$, respectively. The kinematic equations of four players for $t\ge0$ have the form
 \begin{equation}\begin{aligned}\label{dynamics}
\dot{\mathbf{x}}_{P_i}(t)&=\mathbf{p}_i(t),&\mathbf{x}_{P_i}(0)&=\mathbf{x}_{P_i}^0,\quad i=1,2,3\\
\dot{\mathbf{x}}_E(t)&=\mathbf{e}(t),&\mathbf{x}_E(0)&=\mathbf{x}_E^0.
\end{aligned}\end{equation}
 Here, $\mathbf{x}_{P_i}^0=(x_{P_i}^0,y_{P_i}^0,z_{P_i}^0)$ is the initial position of $P_i$, $\mathbf{x}_E^0=(x_E^0,y_E^0,z_E^0)$ is the initial position of $E$, and the control inputs at time $t$ for $P_i$ and $E$ are their respective instantaneous unit headings $\mathbf{p}_i(t)\in\mathcal{U}$ and $\mathbf{e}(t)\in\mathcal{U}$. Thus, the whole state space is $\mathbb{R}^{12}$. Unless for clarity, for simplicity, $t$ will be omitted hereinafter.
 
Three pursuers form as a team, and thus they cooperatively choose their controls. The evasion team has only one member $E$. We consider a non-anticipative information structure, as commonly adopted in the differential game literature (see for example, \cite{Elliott1972The,Mitchell2005A}). Under this information structure, each team has complete up-to-date position information of all players and the control employed by the other team, however, it does not know the control that the other team will apply in the future. Additionally, it is also assumed that four player start the game from different positions and $E$ initially lies in $\Omega_{\rm play}$.

\subsection{Problems}
For this reach-avoid differential game in three-dimensional space with three pursuers and one evader, the following  problem will be addressed. 
\begin{pbm}[Game of kind]\label{kind}
Given $K,b$, and initial configuration $K^\mathsf{T}\mathbf{x}_E^0>b$ and $\mathbf{x}_{P_i}^0\in\mathbb{R}^3$ $(i=1,2,3)$, which team can guarantee its own winning? Does this reach-avoid differential game end up with a successful capture or a successful safe arrival when both team adopt their optimal strategies?
\end{pbm}

\section{Preliminaries}\label{prelinimaries}
We first present some preliminary results, which will be used in the subsequent analysis.
\subsection{Efficient Simplification}
Let $\mathbf{z}=(x,y,z)\in\mathbb{R}^3$. In this section, we describe this game in a more concise and clear way. The TP and two subspaces in (\ref{ori_tar}) can be represented by
\begin{equation}\begin{aligned}\label{new_tar}
&\mathcal{T}=\{\mathbf{z}\in\mathbb{R}^3|z=0\},\Omega_{\rm{goal}}=\{\mathbf{z}\in\mathbb{R}^3|z<0\}\\
&\Omega_{\rm{play}}=\{\mathbf{z}\in\mathbb{R}^3|z>0\}.
\end{aligned}\end{equation}

To simplify the analysis further, \autoref{kind} can be reformulated as follows: Given any initial positions of three pursuers, we aim to find the set of initial positions where if the evader initially lies, three pursuers can guarantee the capture before the evader reaches the TP $\mathcal{T}$, which is the pursuer winning subspace, and find the set of initial positions allowing for a successful safe arrival strategy for the evader, which is the evader winning subspace. The surface, curve or point that separates these two subspaces is the barrier. Fixing three pursuers' initial positions provides a clear illustration of the barrier and thus two wining subspaces, as a function of these initial positions. 

\subsection{Evasion Space}
Let the set of points in $\mathbb{R}^3$ which $E$ can reach before the pursuer(s), regardless of the pursuer(s)' best effort, be called evasion space (ES), and the surface which bounds ES is called the boundary of ES (BES). 

\begin{figure}\centering
\graphicspath{{figure_final/}}
\includegraphics[width=70mm,height=48mm]{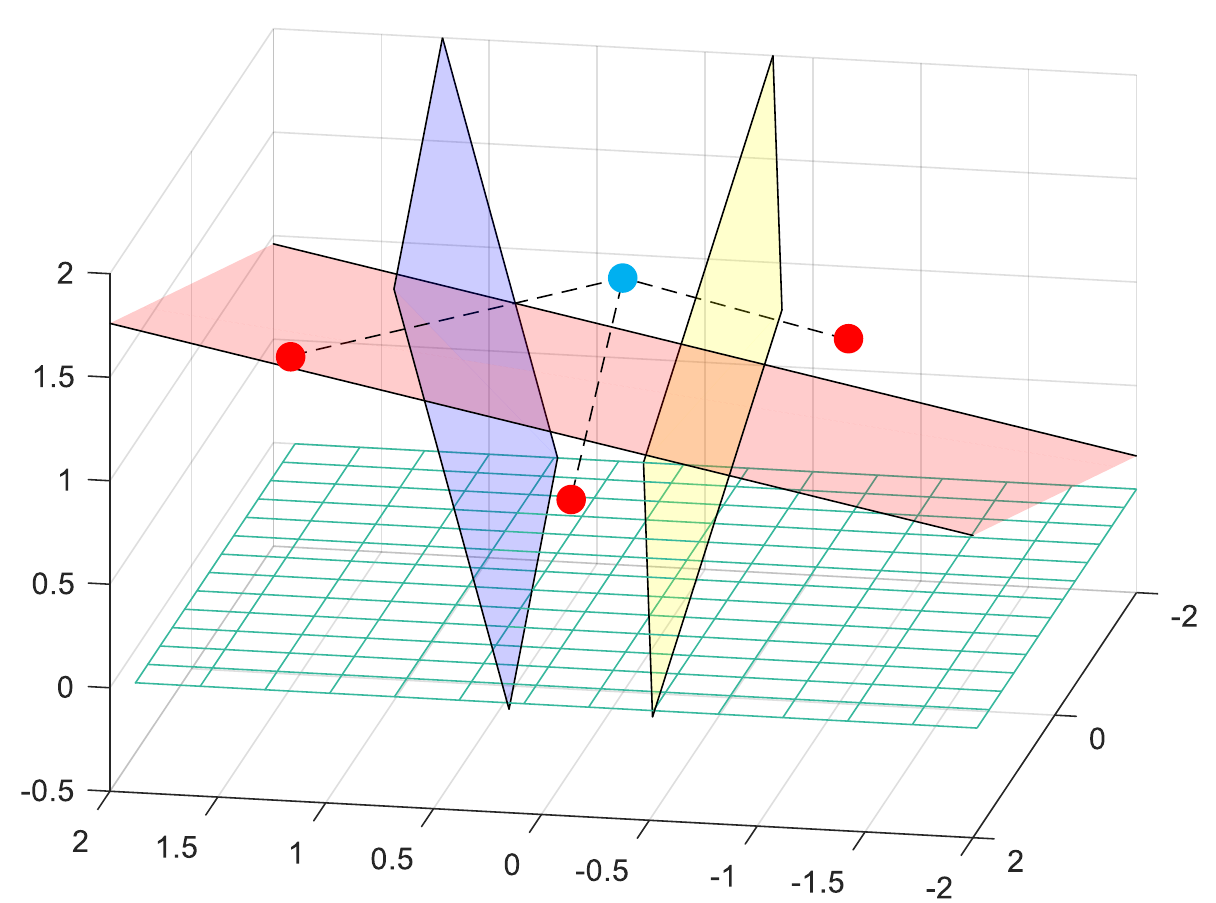}
\put(-163,87){$\scriptstyle{P_1}$}
\put(-104,53){$\scriptstyle{P_2}$}
\put(-57,87){$\scriptstyle{P_3}$}
\put(-100,99){$\scriptstyle{E}$}
\put(-131,104){$\scriptstyle{\mathscr{B}^1}$}
\put(-50,65){$\scriptstyle{\mathscr{B}^2}$}
\put(-83,96){$\scriptstyle{\mathscr{B}^3}$}
\put(-205,55){$\scriptstyle{z}$}
\put(-112,-2){$\scriptstyle{y}$}
\put(-16,20){$\scriptstyle{x}$}
\caption{The evasion space (ES) and the boundary of ES (BES) determined by three pursuers and one evader.\label{fig:88}} 
\end{figure}
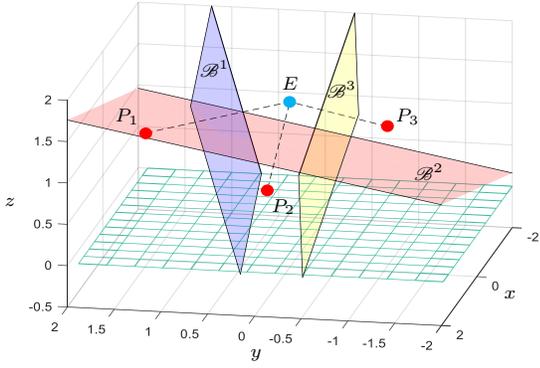

Denote the ES and BES associated with $P_i$ and $E$ at time $t=0$ by $\mathscr{E}^i$ and $\mathscr{B}^i$ respectively. Thus, by definition, the ES and BES can be respectively given as follows:
\begin{equation}\label{AR1}\begin{aligned}
\mathscr{E}^{i}&=\big\{\mathbf{z}\in\mathbb{R}^3|\|\mathbf{z}-\mathbf{x}_E^0\|_2<\|\mathbf{z}-\mathbf{x}_{P_i}^0\|_2\big\}\\
\mathscr{B}^i&=\big\{\mathbf{z}\in\mathbb{R}^3|\|\mathbf{z}-\mathbf{x}_E^0\|_2=\|\mathbf{z}-\mathbf{x}_{P_i}^0\|_2\big\}.
\end{aligned}\end{equation}
As Fig. \ref{fig:88} shows, the ES $\mathscr{E}^i$ and BES $\mathscr{B}^i$ are a half-space and a plane respectively.

Let $\mathscr{E}^{i,j}$ and ${\mathscr{B}}^{i,j}$ denote the ES and BES determined by two pursuers $P_i,P_j$ and evader $E$ at time $t=0$ respectively, which can be obtained by definition as follows:
\begin{equation}\label{AR2}\begin{aligned}
\mathscr{E}^{i,j}=\mathscr{E}^i\cap\mathscr{E}^j,\qquad\mathscr{B}^{i,j}=\partial\mathscr{E}^{i,j}.
\end{aligned}\end{equation}

Similarly, the ES and BES determined by three pursuers $P_1,P_2,P_3$ and evader $E$ at time $t=0$ are denoted by $\mathscr{E}$ and $\mathscr{B}$ respectively, which can also be computed as follows:
\begin{equation}\label{ES3}\begin{aligned}
\mathscr{E}=\mathscr{E}^1\cap\mathscr{E}^2\cap\mathscr{E}^3,\qquad\mathscr{B}=\partial\mathscr{E}.
\end{aligned}\end{equation}
The illustration for ES $\mathscr{E}$ and BES $\mathscr{B}$ is depicted in Fig. \ref{fig:88}.

\section{Barrier and Winning Subspaces }\label{winningsubandbarriersec}
This section focuses on \autoref{kind}, namely, which team will win the game. Obviously, this is a game of kind, which provides a binary answer to the name of game winner, or no one can win the game.

Therefore, the primary goal of this section is to construct the barrier, and subsequently determine their winning subspaces.

Let $\mathcal{B}^i,\mathcal{W}^i_P$ and $\mathcal{W}^i_E$ denote the barrier, pursuer winning subspace and evader winning subspace determined by $P_i$ at time $t=0$ respectively. For two pursuers $P_i$ and $P_j$, let $\mathcal{B}^{i,j},\mathcal{W}^{i,j}_P$ and $\mathcal{W}^{i,j}_E$ denote  the associated barrier, pursuer winning subspace and evader winning subspace at time $t=0$ respectively. Let $\mathcal{B},\mathcal{W}_P$ and $\mathcal{W}_E$ respectively denote the barrier, pursuer winning subspace and evader winning subspace determined by three pursuers together at time $t=0$.

\subsection{One Pursuer Versus One Evader}\label{subsectionb1v1}
We first present the construction of barrier and winning subspaces for the case with one pursuer $P_i$ and one evader $E$, which will provide key insights into the barrier construction for the two-pursuer and three-pursuer scenarios.


\begin{lema}[One pursuer]\rm\label{1v1barrierlema2}
If the system (\ref{dynamics}) has only one pursuer $P_i$, the barrier $\mathcal{B}^i$ has two cases: if $z_{P_i}^0<0$, $\mathcal{B}^i=\big\{\mathbf{z}\in\mathbb{R}^3|x=x_{P_i}^0,y=y_{P_i}^0,z=-z_{P_i}^0\big\}$; otherwise, $\mathcal{B}^i=\emptyset$. Two winning subspaces $\mathcal{W}_P^i$ and $\mathcal{W}_E^i$ are respectively given as follows:
\begin{equation}\begin{aligned}\label{winningsubspaces1vs1v1}
\mathcal{W}_P^i&=\big\{\mathbf{z}\in\mathbb{R}^3|x=x_{P_i}^0,y=y_{P_i}^0,z>|z_{P_i}^0|\big\}\\
\mathcal{W}_E^i&=\Omega_{\rm play}\setminus(\mathcal{B}^i\cup\mathcal{W}_P^i).
\end{aligned}\end{equation}
\end{lema}
\begin{proof} See Fig. \ref{fig:2}(a). Assume that $\mathbf{x}_E^0\in\mathcal{B}^i$. Thus, under $P_i$ and $E$'s optimal strategies, $E$ is captured by $P_i$ exactly when reaching $\mathcal{T}$. Since $\mathscr{B}^i$ is a plane, then $\mathscr{B}^i=\mathcal{T}$ must hold. If $z_{P_i}^0<0$, we can obtain that $\mathbf{x}_E^0$ and $\mathbf{x}_{P_i}^0$ is symmetric with respect to $\mathcal{T}$. If $z_{P_i}^0\ge0$, obviously, $\mathcal{B}^i$ is empty.

If $\mathbf{x}_E^0\in\mathcal{W}_P^i$, $\mathscr{B}^i$ does not intersect with $\mathcal{T}\cup\Omega_{\rm goal}$. Equivalently, $\mathscr{B}^i$ is parallel to $\mathcal{T}$ and lies in $\Omega_{\rm play}$. Thus, $\mathcal{W}_P^i$ is described as (\ref{winningsubspaces1vs1v1}) shows. 

Naturally, $\mathcal{W}_E^i$ is the complementary set of $\mathcal{B}^i\cup\mathcal{W}_P^i$ with respect to $\Omega_{\rm play}$, corresponding to the case in which $\mathscr{B}^i$ intersects with $\Omega_{\rm goal}$. Thus, we finish the proof.
\end{proof}

\begin{figure}\centering
\graphicspath{{figure_final/}}
\subfigure{
\includegraphics[width=41mm,height=32mm]{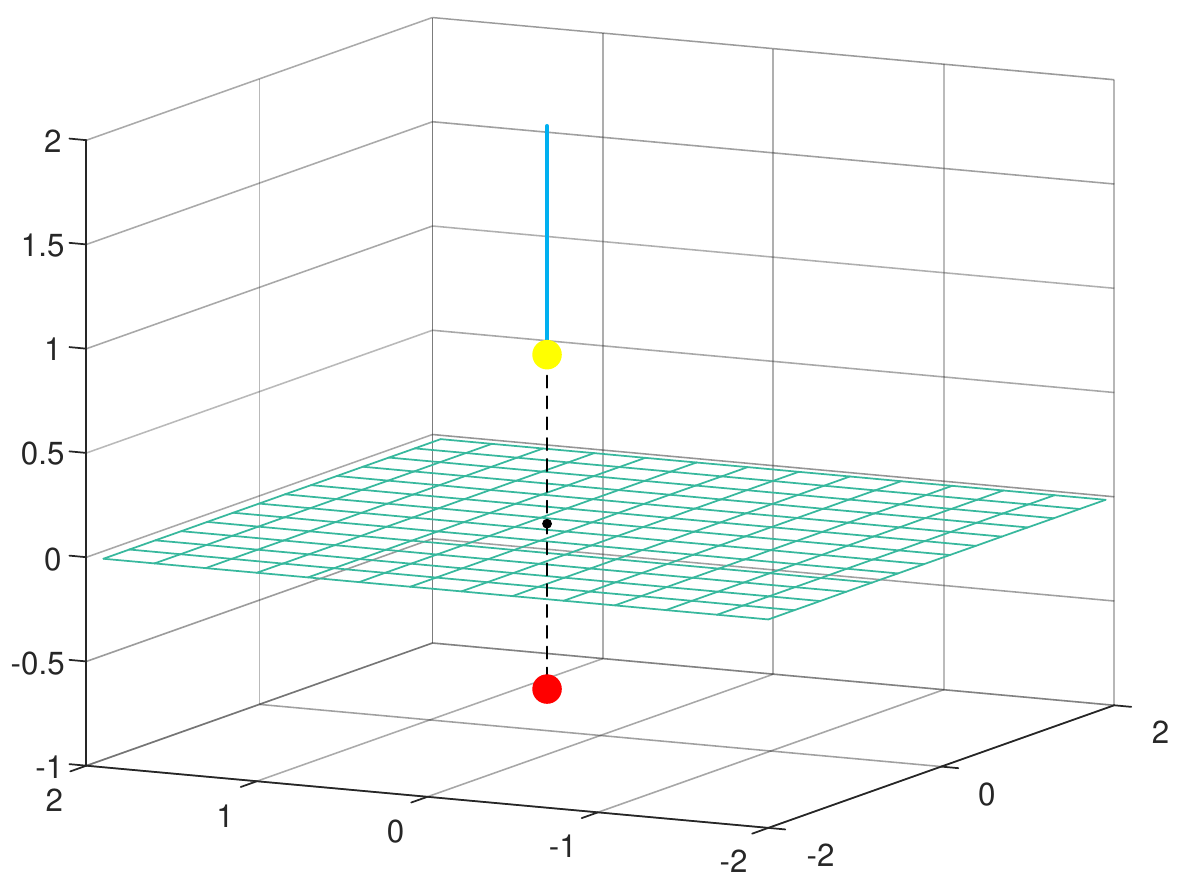}
\put(-32,35){$\scriptstyle{\mathcal{T}}$}
\put(-61,16){$\scriptstyle{P_i}$}
\put(-60,52){$\scriptstyle{\mathcal{B}^i}$}
\put(-61,70){$\scriptstyle{\mathcal{W}^i_P}$}
\put(-28,60){$\scriptstyle{\mathcal{W}^i_E}$}
\put(-121,45){$\scriptstyle{z}$}
\put(-77,0){$\scriptstyle{y}$}
\put(-18,4){$\scriptstyle{x}$}
\put(-70,-10){$\scriptstyle{(a)}$}
}
\subfigure{
\includegraphics[width=41mm,height=32mm]{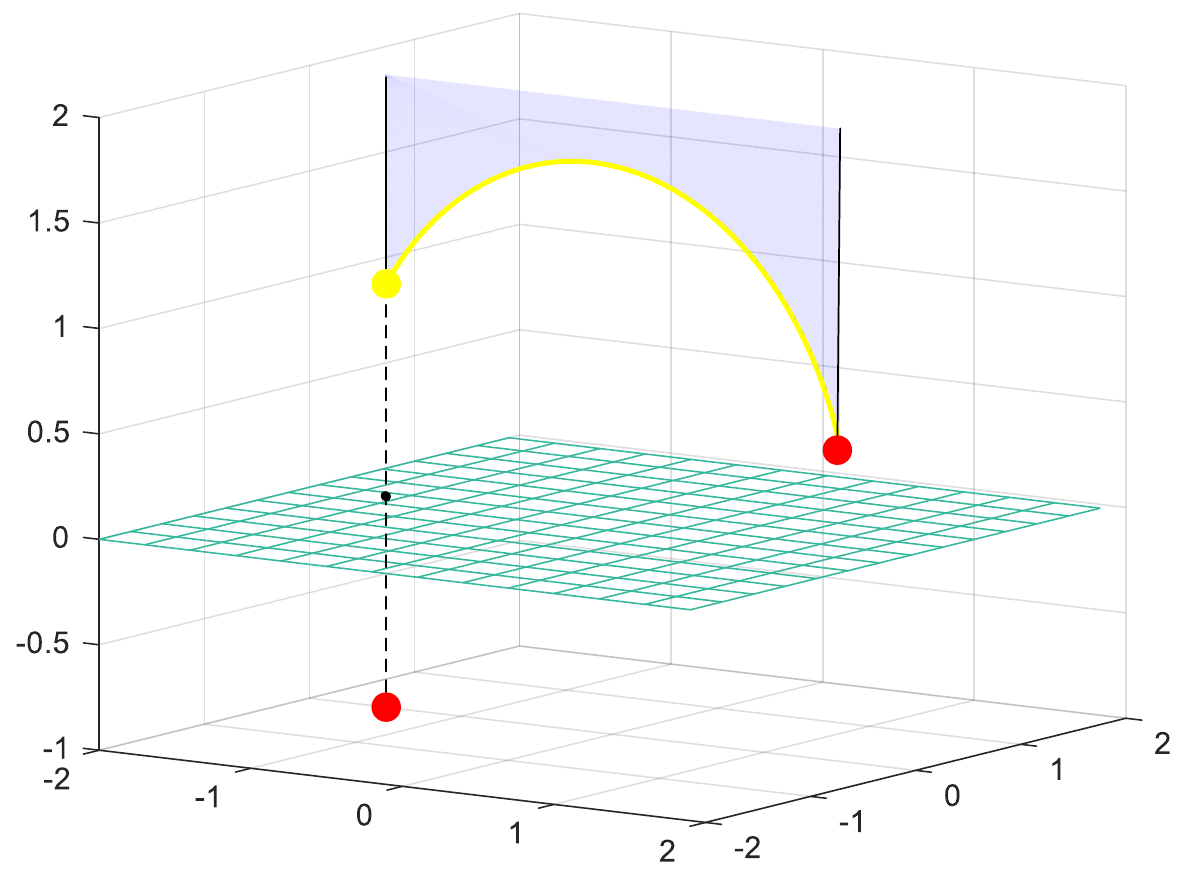}
\put(-48,31){$\scriptstyle{\mathcal{T}}$}
\put(-71,65){$\scriptstyle{\mathcal{B}^{i,j}}$}
\put(-102,50){$\scriptstyle{\mathcal{W}^{i,j}_E}$}
\put(-51,71){$\scriptstyle{\mathcal{W}^{i,j}_P}$}
\put(-90,15){$\scriptstyle{P_i}$}
\put(-33,43){$\scriptstyle{P_j}$}
\put(-121,45){$\scriptstyle{z}$}
\put(-72,0){$\scriptstyle{y}$}
\put(-14,5){$\scriptstyle{x}$}
\put(-70,-10){$\scriptstyle{(b)}$}
}\caption{The barrier and winning subspaces determined by: $(a)$ one pursuer; $(b)$ two active pursuers. The barrier $\mathcal{B}^i$ or $\mathcal{B}^{i,j}$ is in yellow, where $\mathcal{B}^i$ is a point or empty and $\mathcal{B}^{i,j}$ is the part of an arc. The pursuer winning subspace $\mathcal{W}_P^i$ or $\mathcal{W}^{i,j}_P$ is in blue, where $\mathcal{W}_P^i$ is a ray and $\mathcal{W}^{i,j}_P$ is the part of a plane. The evader winning subspace $\mathcal{W}_E^i$ or $\mathcal{W}^{i,j}_E$ is the remainder in $\Omega_{\rm play}$.\label{fig:2}}  
\end{figure}
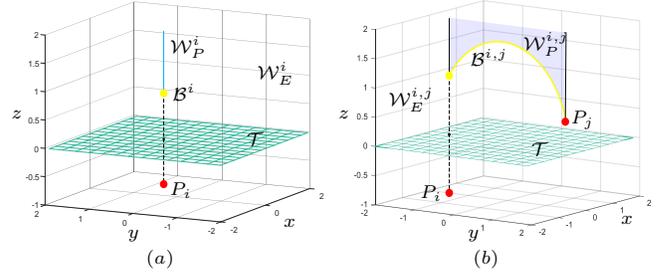

\subsection{Two Pursuers Versus One Evader}\label{subsectionb2v1}
Consider two pursuers $P_i$ and $P_j$. As will be shown below, $\mathcal{B}^{i,j}$ has two types. The first one only depends on one of two pursuers, and the second one is related to both two pursuers. The conditions to distinguish them are as follows. For clarity, the pursuer which the barrier depends on, is called \emph{active pursuer}.

\begin{lema}[Classification condition]\rm\label{twotypeslema} The barrier $\mathcal{B}^{i,j}$ depends on both two pursuers $P_i$ and $P_j$, if and only if
\begin{equation}\label{onefromtwo}\begin{aligned}
(x_{P_i}^0-x_{P_j}^0)^2+(y_{P_i}^0-y_{P_j}^0)^2\neq0.
\end{aligned}\end{equation}
If (\ref{onefromtwo}) fails, $\mathcal{B}^{i,j}$ only depends on one puruser as follows: If $|z_{P_i}^0|<|z_{P_j}^0|$, then $\mathcal{B}^{i,j}=\mathcal{B}^i$, and if $|z_{P_i}^0|>|z_{P_j}^0|$, then $\mathcal{B}^{i,j}=\mathcal{B}^j$, and if $z_{P_i}^0=-z_{P_j}^0$, then $\mathcal{B}^{i,j}=\emptyset$.
\end{lema}
\begin{proof} Note that if $P_i$ can reach any point in $\mathcal{T}$ before $P_j$, then the barrier $\mathcal{B}^{i,j}$ is determined by $P_i$ alone. Thus, the necessary and sufficient condition to determine that $\mathcal{B}^{i,j}$ depends on both two pursuers, is that for each pursuer, there exists at least one point in $\mathcal{T}$ that it can reach before the other pursuer, implying that (\ref{onefromtwo}) holds.

Conversely, if (\ref{onefromtwo}) fails, that is, $x_{P_i}^0=x_{P_j}^0$ and $y_{P_i}^0=y_{P_j}^0$, then the pursuer which is closer to $\mathcal{T}$ determine the barrier $\mathcal{B}^{i,j}$ alone. Additionally, if two pursuers are symmetric with respect to $\mathcal{T}$, it follows from \autoref{1v1barrierlema2} that $\mathcal{B}^{i,j}=\emptyset$. 
\end{proof}

If $\mathcal{B}^{i,j}$ depends on only one pursuer, the results can be obtained from Section \ref{subsectionb1v1}. Thus, we will focus on the case of two active pursuers in which both two pursuers make contributions to the construction of barrier.

Denote by $\bar{\mathbf{x}}_{P_i}^0$ as the projection of $P_i$'s initial position $\mathbf{x}_{P_i}^0$ into $\mathcal{T}$. Thus, $\bar{\mathbf{x}}_{P_i}^0=(x_{P_i}^0,y_{P_i}^0,0)$.

Denote by $\mathcal{L}_{i,j}$ as the straight line in $\mathcal{T}$ and through $\bar{\mathbf{x}}_{P_i}^0$ and $\bar{\mathbf{x}}_{P_j}^0$, and denote the vertical plane through $\mathbf{x}_{P_i}^0,\mathbf{x}_{P_j}^0$ and $\mathcal{L}_{i,j}$ by $\mathcal{P}_{i,j}$. Define $\bm{c}_{i,j}$ as the point in $\mathcal{T}$ that $P_i$ and $P_j$ can reach at the same time. 

\begin{thom}[Two active pursuers]\label{thom2v1equal1}\rm
If the system (\ref{dynamics}) has only two pursuers $P_i$ and $P_j$, and suppose that (\ref{onefromtwo}) is true, then the barrier $\mathcal{B}^{i,j}$ is given by $\cup_{m=1}^3\mathcal{B}^{i,j}_m$ as follows:\begin{equation}\begin{aligned}\label{barrier2v1equal1}
&\mathcal{B}^{i,j}_1=\mathcal{B}^i,\qquad\mathcal{B}^{i,j}_2=\mathcal{B}^j\\
&\mathcal{B}^{i,j}_3=\big\{\mathbf{z}\in\mathbb{R}^3|\|\mathbf{z}-\bm{c}_{i,j}\|_2=\|\mathbf{x}_{P_i}^0-\bm{c}_{i,j}\|_2,\\
&\ \ \ x=\beta x_{P_i}^0+(1-\beta)x_{P_j}^0,y=\beta y_{P_i}^0+(1-\beta)y_{P_j}^0,\\
&\qquad\qquad\qquad\qquad\qquad\qquad\qquad\beta\in(0,1),z>0\big\}
\end{aligned}\end{equation}
and two winning subspaces $\mathcal{W}_P^{i,j}$ and $\mathcal{W}_E^{i,j}$ are respectively given by
\begin{equation}\begin{aligned}\label{winningregions2v1equal1}
&\mathcal{W}^{i,j}_P=\{\mathbf{z}\in\mathbb{R}^3|\|\mathbf{z}-\bm{c}_{i,j}\|_2>\|\mathbf{x}_{P_i}^0-\bm{c}_{i,j}\|_2,\\
&\ \ \ x=\beta x_{P_i}^0+(1-\beta)x_{P_j}^0,y=\beta y_{P_i}^0+(1-\beta)y_{P_j}^0,\\
&\qquad\qquad\qquad\qquad\qquad\qquad\qquad\beta\in[0,1],z>0\}\\
&\mathcal{W}_E^{i,j}=\Omega_{\rm play}\setminus(\mathcal{B}^{i,j}\cup\mathcal{W}_P^{i,j}).
\end{aligned}\end{equation}
\end{thom}
\begin{proof}See Fig. \ref{fig:2}(b). Assume $\mathbf{x}_E^0\in\mathcal{B}^{i,j}$. Thus, $\mathscr{B}^{i,j}$ intersects with $\mathcal{T}$ while does not intersect with $\Omega_{\rm goal}$. Note that $\mathscr{B}^{i,j}$ is the boundary of the intersection set $\mathscr{E}^i\cap\mathscr{E}^j$. Then, it follows from (\ref{onefromtwo}) that the set $\mathscr{B}^i\cap\mathscr{B}^j$ is nonempty.

If the intersection points between $\mathscr{B}^{i,j}$ and $\mathcal{T}$ only belong to $\mathscr{B}^i$, thus $\mathscr{B}^i$ coincides with $\mathcal{T}$. Denote this part of $\mathcal{B}^{i,j}$ by $\mathcal{B}^{i,j}_1$, and \autoref{1v1barrierlema2} leads to $\mathcal{B}^{i,j}_1=\mathcal{B}^i$. Similarly, if these intersection points only belong to $\mathscr{B}^j$, $\mathcal{B}^{i,j}_2=\mathcal{B}^j$ is derived.

For the remainder, when $\mathscr{B}^{i,j}$  intersects with $\mathcal{T}$ at $\mathscr{B}^i\cap\mathscr{B}^j$, then $\mathbf{x}_E^0$ must lie in the vertical plane $\mathcal{P}_{i,j}$. Furthermore, $\mathbf{x}_E^0$ also should lie on the circle of radius $\|\mathbf{x}_{P_i}^0-\bm{c}_{i,j}\|_2$ centered at $\bm{c}_{i,j}$ in this vertical plane. Also note that $\mathbf{x}_E^0$'s projection in $\mathcal{T}$ should lie between 
$\bar{\mathbf{x}}_{P_i}^0$ and $\bar{\mathbf{x}}_{P_j}^0$ along $\mathcal{L}_{i,j}$. Denote this part of $\mathcal{B}^{i,j}$ by $\mathcal{B}^{i,j}_3$. Thus, $\mathcal{B}^{i,j}_3$ can be mathematically described as (\ref{barrier2v1equal1}) shows. 

As for two winning subspaces, attention will be focused on $\mathcal{W}_P^{i,j}$, and $\mathcal{W}_E^{i,j}$ is naturally the remainder as (\ref{winningregions2v1equal1}) shows. 

Assume that $\mathbf{x}_E^0\in\mathcal{W}_P^{i,j}$. Thus, $\mathscr{B}^{i,j}$ will not intersect with the set $\mathcal{T}\cup\Omega_{\rm goal}$. It can be observed that in this case $\mathbf{x}_E^0$ must lie in the vertical plane $\mathcal{P}_{i,j}$; otherwise, $\mathscr{B}^{i,j}$ intersects with the set $\mathcal{T}\cup\Omega_{\rm goal}$. Additionally, $\mathbf{x}_E^0$ also should lie ouside the circle on which the barrier lies, and its projection in $\mathcal{T}$ should lie between $\bar{\mathbf{x}}_{P_i}^0$ and $\bar{\mathbf{x}}_{P_j}^0$ along $\mathcal{L}_{i,j}$. In addition, $\mathbf{x}_E^0\in\mathcal{W}_P^i$ or $\mathbf{x}_E^0\in\mathcal{W}_P^j$ is satisfactory. Thus, $\mathcal{W}_P^{i,j}$ is as (\ref{winningregions2v1equal1}) shows.
\end{proof}


\subsection{Three Pursuers Versus One Evader}\label{subsectionb3v1}
As will be shown below, $\mathcal{B}$ has three types. The first one is only dependent on one of three pursuers, the second one is associated with two of them, and the third one depends on all of them. The conditions to distinguish three types are as follows. Define two index sets $\mathcal{I}_i=\{1,2,3\}\setminus\{i\}$ and $\mathcal{I}_{i,j}=\{1,2,3\}\setminus\{i,j\}$. Denote by $\hat{\mathbf{x}}_{P_i}^0$ the symmetric point of $\mathbf{x}_{P_i}^0$ with respect to $\mathcal{T}$. 

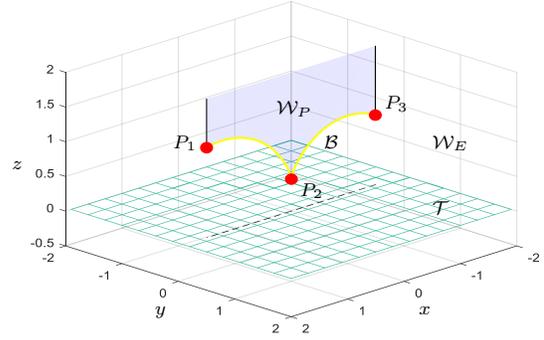
\begin{figure}\centering
\graphicspath{{figure_final/}}
\includegraphics[width=70mm,height=45mm]{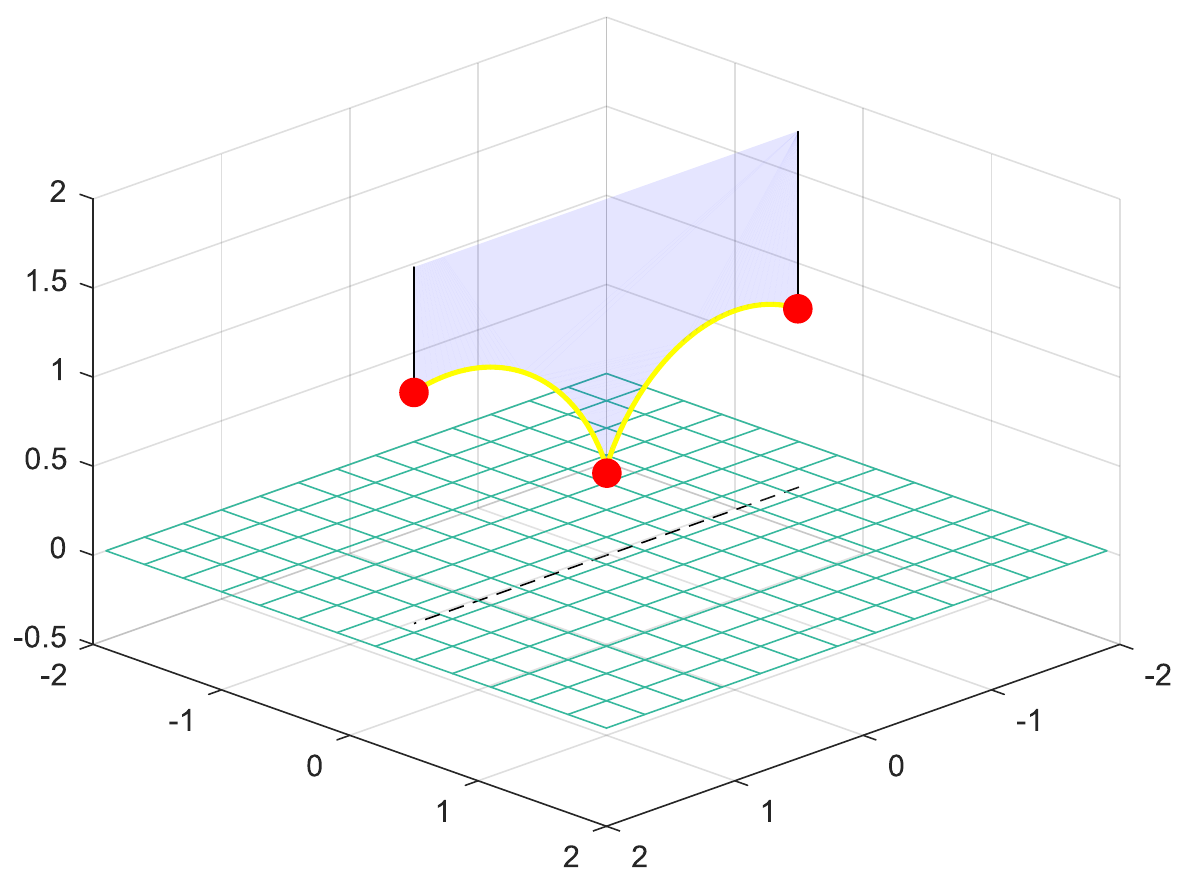}
\put(-143,70){$\scriptstyle{P_1}$}
\put(-95,52){$\scriptstyle{P_2}$}
\put(-63,85){$\scriptstyle{P_3}$}
\put(-86,70){$\scriptstyle{\mathcal{B}}$}
\put(-104,83){$\scriptstyle{\mathcal{W}_P}$}
\put(-45,70){$\scriptstyle{\mathcal{W}_E}$}
\put(-45,45){$\scriptstyle{\mathcal{T}}$}
\put(-204,62){$\scriptstyle{z}$}
\put(-150,7){$\scriptstyle{y}$}
\put(-50,7){$\scriptstyle{x}$}
\caption{The barrier and winning subspaces for three collinear-projection active pursuers. The barrier $\mathcal{B}$ in yellow consists of two arcs, and the pursuer winning subspace $\mathcal{W}_P$ in blue is the part of a plane. The evader winning subspace $\mathcal{W}_E$ is the remainder in $\Omega_{\rm play}$.  \label{fig:3}} 
\end{figure}

\begin{lema}[Classification condition]\rm\label{threetypeslema} The barrier $\mathcal{B}$ only depends on the pursuer $P_i$, namely, $\mathcal{B}=\mathcal{B}^i$, if and only if 
\begin{equation}\label{onefromthree}\begin{aligned}
x_{P_i}^0=x_{P_k}^0,\quad y_{P_i}^0=y_{P_k}^0,\quad |z_{P_i}^0|<|z_{P_k}^0|
\end{aligned}\end{equation}
holds for all $k\in\mathcal{I}_i$. The barrier $\mathcal{B}$ depends on only two pursuers $P_i$ and $P_j$, namely, $\mathcal{B}=\mathcal{B}^{i,j}$, if and only if 
\begin{equation}\label{twofromthree}\begin{aligned}
(x_{P_i}^0-x_{P_j}^0)^2+(y_{P_i}^0-y_{P_j}^0)^2\neq0,\{\mathbf{x}_{P_k}^0,\hat{\mathbf{x}}_{P_k}^0\}\cap\mathcal{W}_P^{i,j}\neq\emptyset\\
\text{or }x_{P_i}^0=x_{P_k}^0,y_{P_i}^0=y_{P_k}^0,z_{P_i}^0=-z_{P_k}^0>0,\mathbf{x}_{P_i}^0\notin\mathcal{W}_P^{j}
\end{aligned}\end{equation}
where $k\in\mathcal{I}_{i,j}$. The barrier $\mathcal{B}$ depends on all three pursuers, if and only if 
\begin{equation}\label{threefromthree}\begin{aligned}
(x_{P_i}^0-x_{P_j}^0)^2+(y_{P_i}^0-y_{P_j}^0)^2\neq0,\{\mathbf{x}_{P_k}^0,\hat{\mathbf{x}}_{P_k}^0\}\cap\mathcal{W}_P^{i,j}=\emptyset
\end{aligned}\end{equation}
holds for all $i,j\in\{1,2,3\},i\neq j$ and $k\in\mathcal{I}_{i,j}$. 
\end{lema}
\begin{proof}Note that the barrier $\mathcal{B}$ only depends on the pursuer $P_i$, if and only if $P_i$ can reach any point in $\mathcal{T}$ before the other two pursuers. Thus, (\ref{onefromthree}) holds for all $k\in\mathcal{I}_i$.

Consider the case in which the barrier $\mathcal{B}$ only depends on two pursuers $P_i$ and $P_j$. Firstly, \autoref{twotypeslema} implies that (\ref{onefromtwo}) is true, and for $k\in\mathcal{I}_{i,j}$, there exist no points in $\mathcal{T}$ that the puruser $P_k$ can reach before both $P_i$ and $P_j$, meaning that $\mathbf{x}_{P_k}^0\in\mathcal{W}_P^{i,j}$ or $\hat{\mathbf{x}}_{P_k}^0\in\mathcal{W}_P^{i,j}$. Secondly, when $P_{i}$ and $P_k$ are symmetric with respect to $\mathcal{T}$, $P_i$ should lie in $\Omega_{\rm play}$ and out of $\mathcal{W}_P^j$. Thus, the condition (\ref{twofromthree}) is obtained.

As for the case in which the barrier $\mathcal{B}$ depends on all three pursuers, for each pursuer, there exists at least one point in $\mathcal{T}$ that it can reach prior to the other two pursuers. Thus, it follows from \autoref{twotypeslema} that (\ref{onefromtwo}) holds for all $i,j\in\{1,2,3\}$ and $i\neq j$. Additionally, for $k\in\mathcal{I}_{i,j}$, $\mathbf{x}_{P_k}^0\notin\mathcal{W}_P^{i,j}$ and $\hat{\mathbf{x}}_{P_k}^0\notin\mathcal{W}_P^{i,j}$ are satisfied.
\end{proof}

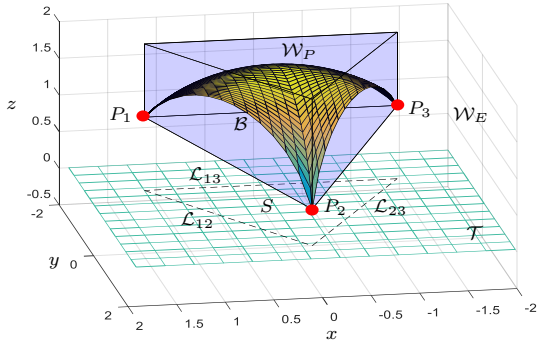
\begin{figure}\centering
\graphicspath{{figure_final/}}
\includegraphics[width=70mm,height=45mm]{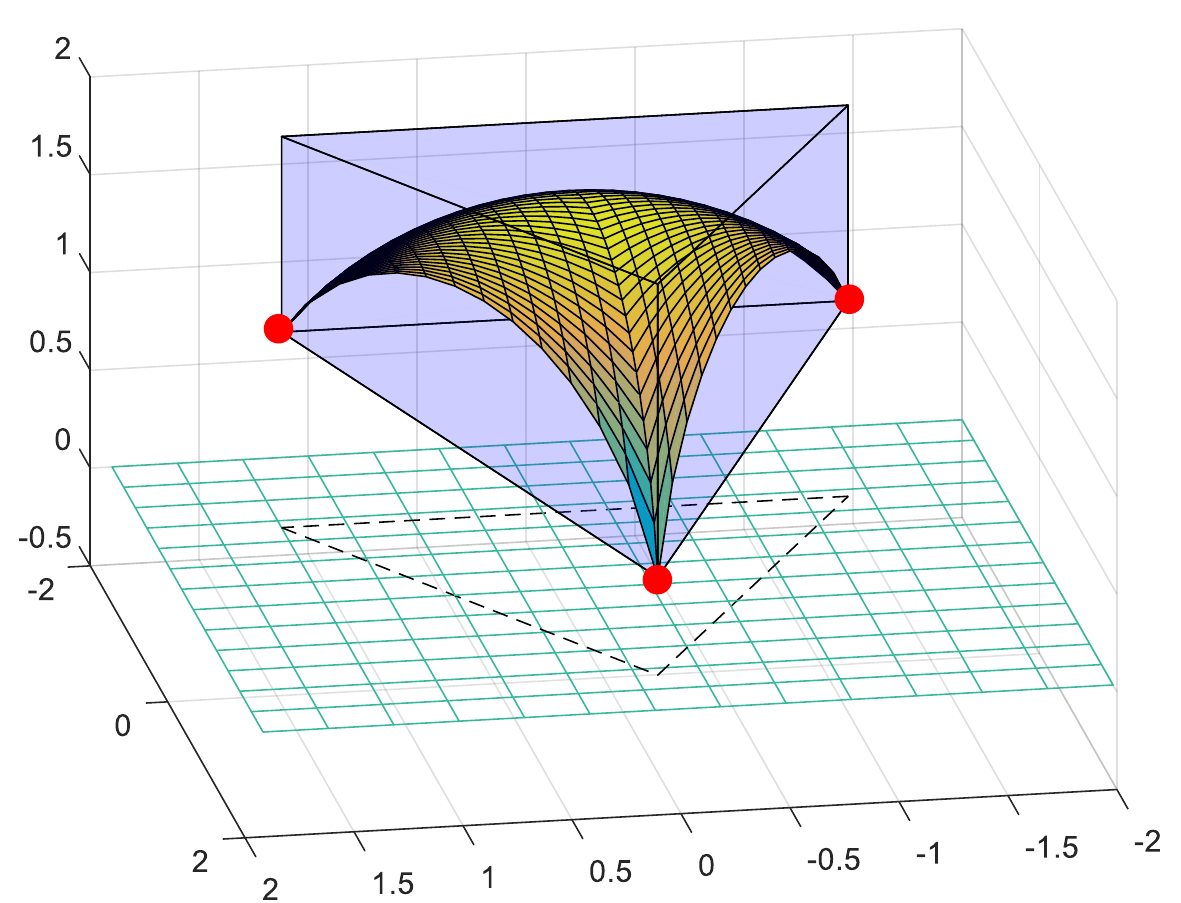}
\put(-165,80){$\scriptstyle{P_1}$}
\put(-84,45){$\scriptstyle{P_2}$}
\put(-52,82){$\scriptstyle{P_3}$}
\put(-108,45){$\scriptstyle{S}$}
\put(-138,40){$\scriptstyle{\mathcal{L}_{12}}$}
\put(-135,57){$\scriptstyle{\mathcal{L}_{13}}$}
\put(-65,45){$\scriptstyle{\mathcal{L}_{23}}$}
\put(-118,76){$\scriptstyle{\mathcal{B}}$}
\put(-100,104){$\scriptstyle{\mathcal{W}_P}$}
\put(-35,80){$\scriptstyle{\mathcal{W}_E}$}
\put(-30,35){$\scriptstyle{\mathcal{T}}$}
\put(-204,84){$\scriptstyle{z}$}
\put(-188,25){$\scriptstyle{y}$}
\put(-83,-3){$\scriptstyle{x}$}
\caption{The barrier and winning subspaces for three noncollinear-projection active pursuers. The barrier $\mathcal{B}$ in yellow is the part of a sphere, and the pursuer winning subspace $\mathcal{W}_P$ is the region above $\mathcal{B}$ and bounded by a triangular prism in blue. The evader winning subspace $\mathcal{W}_E$ is the remainder in $\Omega_{\rm play}$, and $S$ is a closed region in $\mathcal{T}$ bounded by a triangle with three pursuers' projections in $\mathcal{T}$ as its vertexes. The dashed line $\mathcal{L}_{i,j}$ is the projection line of two pursuers $P_i$ and $P_j$ in $\mathcal{T}$.\label{fig:4}} 
\end{figure}

\begin{thom}[Three collinear-projection active pursuers]\label{thom3v1collinear}\rm
Consider the system (\ref{dynamics}) and suppose that (\ref{threefromthree}) is true. If the projections of three pursuers' initial positions into $\mathcal{T}$ are collinear and $P_i$'s projection $\bar{\mathbf{x}}_{P_i}^0$ sits in the middle, then the barrier $\mathcal{B}$ and two winning subspaces $\mathcal{W}_P$ and $\mathcal{W}_E$ are respectively given by
\begin{equation}\begin{aligned}\label{winningregion3v1equal1}
\mathcal{B}=\bigcup_{k\in\mathcal{I}_i}\mathcal{B}^{i,k},\mathcal{W}_P=\bigcup_{k\in\mathcal{I}_i}\mathcal{W}_P^{i,k},\mathcal{W}_E=\bigcap_{k\in\mathcal{I}_i}\mathcal{W}_E^{i,k}.
\end{aligned}\end{equation}
\end{thom}
\begin{proof}See Fig. \ref{fig:3}. Since (\ref{threefromthree}) is true and three pursuers' projections in $\mathcal{T}$ are collinear with $\bar{\mathbf{x}}_{P_i}^0$ in the middle, there is no point in $\mathcal{T}$ that both two pursuers in $\mathcal{I}_i$ can reach before $P_i$. Thus, the barrier $\mathcal{B}$ consists of two parts $\mathcal{B}^{i,k}$ with $k\in\mathcal{I}_i$, where we take $i=2$ in Fig. \ref{fig:3}.

It can be observed that the winning subspace $\mathcal{W}_P$ is the union set of $\mathcal{W}_P^{i,k}$ with $k\in\mathcal{I}_i$. The winning subspace $\mathcal{W}_E$ is naturally the remainder in $\Omega_{\rm play}$, and it can also be written as the intersection set of $\mathcal{W}_E^{i,k}$ with $k\in\mathcal{I}_i$.
\end{proof}

Next, we consider the case in which $\bar{\mathbf{x}}_{P_i}^0$ $(i=1,2,3)$ are not collinear. Denote by $S$ the closed region in $\mathcal{T}$ bounded by a triangle with three projections $\bar{\mathbf{x}}_{P_i}^0$ $(i=1,2,3)$ as its vertexes. Denote by $\bm{c}$ the unique point in $\mathcal{T}$ that has the equal distance to three pursuers. 

\emph{\textbf{Theorem} 3} \emph{(Three noncollinear-projection active pursuers):} Consider the system (\ref{dynamics}) and suppose that (\ref{threefromthree}) is true. If the projections of three pursuers' initial positions into $\mathcal{T}$ are not collinear, the barrier $\mathcal{B}$ is given as follows:
\begin{equation}\begin{aligned}\label{barrier3v1colinear1}
\mathcal{B}=\big\{\mathbf{z}\in\mathbb{R}^3|\|\mathbf{z}-\bm{c}\|_2=\|\mathbf{x}_{P_1}^0-\bm{c}\|_2,(x,y)\in S,z>0\big\}
\end{aligned}\end{equation}
and two winning subspaces $\mathcal{W}_P$ and $\mathcal{W}_E$ are respectively given by
\begin{equation}\begin{aligned}\label{winningregions3v1uncolinear}
\mathcal{W}_P&=\big\{\mathbf{z}\in\mathbb{R}^3|\|\mathbf{z}-\bm{c}\|_2>\|\mathbf{x}_{P_1}^0-\bm{c}\|_2,(x,y)\in S,z>0\big\}\\
\mathcal{W}_E&=\Omega_{\rm play}\setminus(\mathcal{B}\cup\mathcal{W}_P).
\end{aligned}\end{equation}

\begin{proof}See Fig. \ref{fig:4}. Assume that $\mathbf{x}_E^0\in\mathcal{B}$. Thus, $\mathscr{B}$ intersects with $\mathcal{T}$ while does not intersect with $\Omega_{\rm goal}$. Since $\mathscr{B}$ is the boundary of the intersection set $\mathscr{E}^1\cap\mathscr{E}^2\cap\mathscr{E}^3$, thus $\mathscr{B}$ is made up of parts of $\mathscr{B}^i$ $(i=1,2,3)$.

If the intersection points between $\mathscr{B}$ and $\mathcal{T}$ only belong to $\mathscr{B}^i$, we have $\mathbf{x}_E^0\in\mathcal{B}^i$. If these intersection points only belong to the set $\mathscr{B}^i\cap\mathscr{B}^j$, we have $\mathbf{x}_E^0\in\mathcal{B}^{i,j}$.

Finally, we consider the case in which these intersection points belong to the set $\mathscr{B}^1\cap\mathscr{B}^2\cap\mathscr{B}^3$. Since (\ref{threefromthree}) is true and three pursuers' projections $\bar{\mathbf{x}}_{P_i}^0$ $(i=1,2,3)$ are not collinear, then $\mathscr{B}^1\cap\mathscr{B}^2\cap\mathscr{B}^3$ has the unique element, i.e., the point $\bm{c}$. Thus, $\mathbf{x}_E^0$ should lie on the sphere of radius $\|\mathbf{x}_{P_1}^0-\bm{c}\|_2$ centered at $\bm{c}$. Additionally, $\mathscr{E}=\mathscr{E}^1\cap\mathscr{E}^2\cap\mathscr{E}^3$ guarantees that $\mathbf{x}_E^0$'s projection in $\mathcal{T}$ should lie inside the triangle $S$. Thus, the barrier $\mathcal{B}$ is derived and can be described as (\ref{barrier3v1colinear1}) shows, depicted in Fig. \ref{fig:4}.

Similarly, we only consider the winning subspace $\mathcal{W}_P$. If $\mathbf{x}_E^0\in\mathcal{W}_P$, $\mathscr{B}$ does not intersect with $\mathcal{T}\cup\Omega_{\rm goal}$. First, $\mathbf{x}_E^0$'s projection in $\mathcal{T}$ lies in $S$. Otherwise, assume that $\mathbf{x}_E^0$'s projection lies outside $S$ and in the opposite side of $\bar{\mathbf{x}}_{P_k}^0$ $(k\in\mathcal{I}_{i,j})$ with respect to $\mathcal{L}_{i,j}$. Then, there always exist points in $\mathcal{T}$ far away from $\mathcal{L}_{i,j}$ along the side of $\mathbf{x}_E^0$'s projection, such that $E$ can reach before all three pursuers.

Since $\mathbf{x}_E^0$'s projection in $\mathcal{T}$ lies in $S$, we can conclude that $\mathbf{x}_E^0$ should lie outside the sphere defined above. Thus, the winning subspace $\mathcal{W}_P$ is given by (\ref{winningregions3v1uncolinear}).
\end{proof}

\begin{figure}\centering
\graphicspath{{figure_final/}}
\includegraphics[width=70mm,height=48mm]{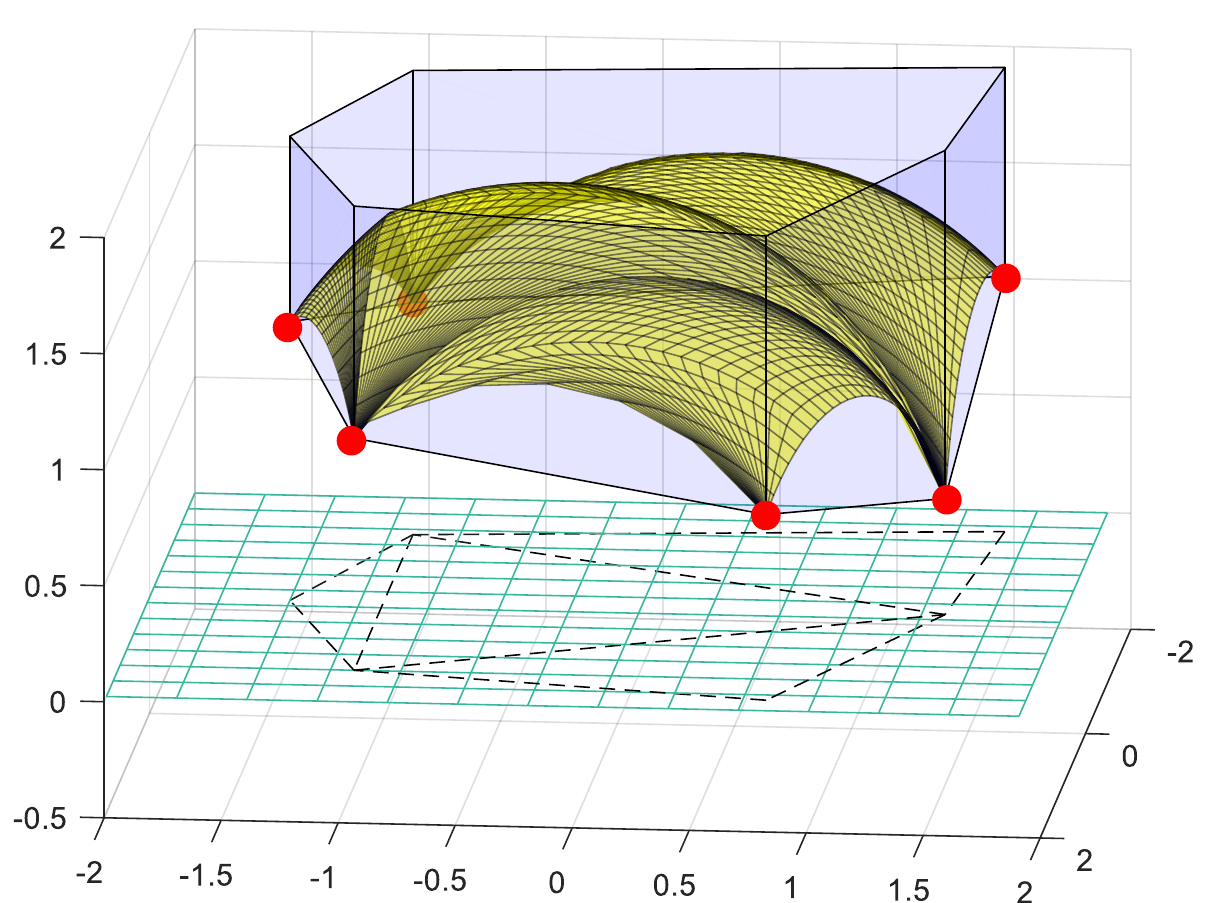}
\put(-163,81){$\scriptstyle{P_1}$}
\put(-151,63){$\scriptstyle{P_2}$}
\put(-85,52){$\scriptstyle{P_3}$}
\put(-40,60){$\scriptstyle{P_4}$}
\put(-31,92){$\scriptstyle{P_5}$}
\put(-141,108){$\scriptstyle{P_6}$}
\put(-110,41){$\scriptstyle{S}$}
\put(-100,88){$\scriptstyle{\mathcal{B}}$}
\put(-100,118){$\scriptstyle{\mathcal{W}_P}$}
\put(-28,115){$\scriptstyle{\mathcal{W}_E}$}
\put(-42,30){$\scriptstyle{\mathcal{T}}$}
\put(-204,60){$\scriptstyle{z}$}
\put(-103,-4){$\scriptstyle{y}$}
\put(-9,15){$\scriptstyle{x}$}
\caption{The barrier and winning subspaces for six noncollinear-projection active pursuers. The barrier $\mathcal{B}$ in yellow consists of the part of four spheres, and the pursuer winning subspace $\mathcal{W}_P$ is the region above $\mathcal{B}$ and bounded by a hexagonal prism in blue. The evader winning subspace $\mathcal{W}_E$ is the remainder in $\Omega_{\rm play}$, and $S$ is a closed region in $\mathcal{T}$ bounded by a hexagon with six pursuers' projections in $\mathcal{T}$ as its vertexes.\label{fig:5}} 
\end{figure}
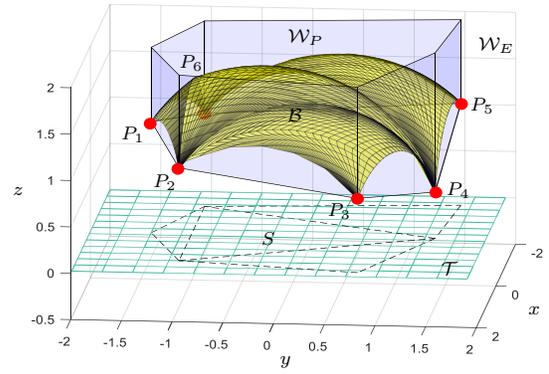
\subsection{Extensions to Multiple Pursuers Case}
Next, we extend the results in former sections into multiple pursuers case. An algorithm is provided here to construct the barrier and winning subspaces for multiple pursuers, shown in Algorithm 1.

\begin{algorithm}
\SetAlgoLined
\KwResult{$\mathcal{B},\mathcal{W}_P$ and $\mathcal{W}_E$ for $n$ pursuers}
 \KwData{$n$ pursuers' initial positions, $\mathcal{T},\Omega_{\rm play}$ and $\Omega_{\rm goal}$}
 \For{each of $n$ pursuers}{
  \eIf{there exists a point that it can reach before the other $n-1$ pursuers}{
		this pursuer is an active pursuer\;
   }{
   this pursuer is not an active pursuer\;
  }
 }
  \eIf{there exist less than three active pursuers}{
		\ construct the barrier and winning subspaces by \autoref{1v1barrierlema2} or \autoref{thom2v1equal1}\;
   }{
   \For{all active pursuers}{
\  find all active triple-pursuer coalitions, in each of which three pursuers have a closer distance to the unique point in $\mathcal{T}$ that they can reach at the same time, before the other active pursuers\;  
 }
\For{all active triple-pursuer coalitions}{
\ construct the barrier and winning subspaces by \autoref{thom3v1collinear} or Theorem 3\;}
\ compute the union sets for all barriers and pursuer winning subspaces, and the intersection set for all evader winning subspaces\;
  }
 
\caption{Barrier and winning subspaces for multiple pursuers\label{al1}}
\end{algorithm}

In Fig. \ref{fig:5}, six active pursuers $P_i$ $(i=1,...,6)$ are considered, which consist of four active triple-pursuer coalitions defined in Algorithm 1: $\{P_1,P_2,P_6\}$, $\{P_2,P_3,P_4\}$, $\{P_2,P_4,P_6\}$ and $\{P_4,P_5,P_6\}$. By computing the barrier and winning subspaces for each active triple-pursuer coalition, the barrier $\mathcal{B}$ and pursuer winning subspace $\mathcal{W}_P$ for the pursuer team are their union sets, and the evader winning subspace $\mathcal{W}_E$ is their intersection set.

Thus, it can be observed from Fig. \ref{fig:5} that multiple pursuers form a net to guarantee that the evader cannot reach the goal region. This result also provides a potential way to design the pursuer placement such that the area of the net is maximized.

\section{Conclusion}\label{conclusion}
A three-pursuer-one-evader reach-avoid differential game in three-dimensional space has been addressed from kind in this paper. By discussing all possible cooperations among pursuers and computing the associated barrier and winning subspaces analytically, which team has a guaranteed winning strategy can be determined before the game evolves. It has also been shown that different relative positions among pursuers result in the barrier and winning subspaces of different shape. The results are also extended to multiple pursuers case, and it has been shown that the pursuers form a net consisting of parts of several spheres to prevent the evader from entering the goal region. Future work will focus on the quantitative problems, such as how to capture the evader within the minimal time. 




\bibliographystyle{IEEEtran}
\bibliography{references}

\end{document}